\documentclass[12pt,twoside]{amsart}
\usepackage{amssymb}
\usepackage{a4wide}
\usepackage[latin1]{inputenc}
\usepackage[T1]{fontenc}
\usepackage{times}

\def\hpq0{h^{p,q}_{\leq 0}}
\def\Hpq0{\H_{\leq 0}^{p,q}}

\def\H{{\mathcal H}}

\def\be{\begin{equation}}
\def\ee{\end{equation}}

\newtheorem{thm}{Theorem}[section]

\newtheorem{prop}[thm]{Proposition}

\theoremstyle{definition}

\theoremstyle{remark}

\newtheorem{preremark}{Remark}
\newtheorem{preex}{Example}

\numberwithin{equation}{section}

\title[]
{A comparison principle for Bergman kernels.}

\address{Department of Mathematics\\Chalmers University
  of Technology \\
 S-412 96
  G\"OTEBORG\\SWEDEN} 

\email{ bob@chalmers.se}

\author[]{Bo Berndtsson}

\begin{document}

\begin{abstract}
We give a version of the comparison principle from pluripotential theory where the Monge-Amp\`ere measure is replaced by the Bergman kernel and use it to derive a maximum principle.
\end{abstract}

{\it Dedicated to the memory of Mikael Passare}

\bigskip

\maketitle

\section{Introduction.}

Let $\phi$ and $\psi$ be two plurisubharmonic functions in a complex manifold $X$, and let $\Omega$ be a relatively compact subdomain in $X$. Assume that on the boundary of $\Omega$, $\phi\leq\psi$, and that inside the domain the Monge-Ampere measures of $\phi$ and $\psi$ satisfy
$$
(dd^c\phi)^n\geq (dd^c\psi)^n.
$$
Then the maximum principle for the Monge-Ampere equation asserts that the inequality $\phi\leq \psi$ holds inside the domain $\Omega$ too. (Here of course both the inequality between $\phi$ and $\psi$ on the boundary and the Monge-Ampere equation has to be given a precise meaning.) The maximum principle is easy to prove if the functions are sufficiently smooth, e g of class $\mathcal{C}^2$. For non regular functions the maximum principle can be derived from the so called comparison principle (see \cite{Kol}) of Bedford and Taylor, which also serves as a substitute for the maximum principle in some cases.  The comparison principle  states (again omitting precise assumptions) that
$$
\int_{\{\psi<\phi\}}(dd^c\phi)^n\leq \int_{\{\psi<\phi\}}(dd^c\psi)^n .
$$

On the other hand it is well known that Monge-Amp\`ere measures often can be approximated by measures defined by {\it Bergman functions}.   Suppose that we have given on our manifold $X$ a positive measure, $\mu$, and consider the $L^2$-space of holomorphic functions 
$$
A^2=A^2(X,\mu,\phi)=\{h\in H(X); \int |h|^2e^{-\phi} d\mu <\infty\},
$$
or its closure in $L^2(X,\mu,\phi)$.
We denote by $K_\phi(z,\zeta)$ the Bergman kernel for $A^2$ and let 
$$
B_\phi(z)=K_\phi(z,z)e^{-\phi}
$$
be the  Bergman function, also known as the {\it density of states function}. It is a consequence of the asymptotic expansion formula of Tian-Catlin-Zelditch (see \cite{Zeld}) that we have
$$
\lim_{k\rightarrow \infty} k^{-n}B_{k\phi} d\mu =c_n (dd^c\phi)^n
$$
if $\phi$ is plurisubharmonic and $\phi$ and $\mu$ are sufficiently regular. We can therefore  think of  $B_\phi d\mu$ as an approximation, or perhaps quantization,  of the Monge-Ampere measure of $\phi$. 

The main observation in this note is that a version of  the comparison principle  holds if we replace the Monge-Ampere operator by the density of states function, so that   
$$
\int_{\psi<\phi} B_\phi d\mu\leq \int_{\psi<\phi} B_\psi d\mu.
$$
As it turns out, this is an almost completely formal phenomenon, and it holds under very (but not completely) general circumstances. In particular, the plurisubharmonicity of $\phi$ and $\psi$ plays no role at all, and even the holomorphicity of functions in $A^2$ enters only in a very weak form, so similar results also  hold in many other situations when we have a well behaved Bergman kernel, and  also if we consider sections of line bundles instead av scalar valued functions. However, the setting of plurisubharmonic weights and holomorphic functions allows a slightly stronger statement with strict inequality, and in that context our main theorem is as follows.
\begin{thm} Let $L$ be a holomorphic line bundle over a complex manifold $X$, and let $\phi$ and $\psi$ be two, possibly singular, metrics on $L$. Suppose that $dd^c\phi\geq -\omega$ and $dd^c\psi\geq -\omega$ for some smooth hermitean $(1,1)$-form $\omega$. Assume also that for some constant $C$, $\phi\leq \psi +C$ and that $\mu$ is given by a strictly positive continuous volume form.
Then
\be
\int_{\psi<\phi} B_\phi d\mu\leq \int_{\psi<\phi} B_\psi d\mu.
\ee
Moreover, if $\emptyset\neq \{\psi<\phi\}\neq X$, strict inequality holds unless both sides are zero or infinity. 
\end{thm}
A few remarks are in order. The strict inequality is of less importance when we deal with Monge-Amp\`ere measures, since one can often arrange that by an ad hoc small perturbation. For Bergman kernels this is less clear and that is the reason why we mention the (very weak) conditions for strict inequality. The condition that $\phi\leq \psi+C$ is sometimes phrased as '$\psi$ is less singular than $\phi$', and some condition like that is necessary. Indeed, if $\psi<\phi$ everywhere and we assume $X$ compact, the two integrals equal the dimensions of the space of sections of $L$ that are square integrable
with respect to the respective metrics. If $\psi$ is more singular than $\phi$ it may well happen that the space of sections that have finite norm measured by $\psi$ is smaller than the space of sections that have finite norm measured by $\phi$, so the inequality cannot hold.

\section{The abstract setting}
We will first deal with the abstract setting of general $L^2$-spaces with a Bergman kernel. Let $(X, \mu)$ be a measure space, let $e^{-\phi}$ be a measurable weight function on $X$,  and let $\H_\phi$ be a Hilbert subspace of  $L^2(e^{-\phi}d\mu)$. We assume that for any $z\in X$,  point evaluation at $z$ is a bounded linear functional on $\H_\phi$. Then $\H_\phi$ has a Bergman kernel, $K_\phi(z,\zeta)$ and we denote $B_\phi(z)=K_\phi(z,z)e^{-\phi}$.

\begin{thm}(Comparison principle for Bergman spaces.)  Let $\phi$ and $\psi$ be two weight functions on $X$ such that for some constant $C$, $\phi\leq \psi +C$.  Then 
\be
\int_{\psi<\phi} B_\phi d\mu \leq\int_{\psi<\phi} B_\psi d\mu.
\ee
\end{thm}
To prove the comparison principle  we need a, basically standard,  lemma on derivatives of Bergman kernels. 
\begin{prop} Let $\phi_t$ be a differentiable family of weight functions with uniformly bounded derivative with respect to $t$ . Put $K_t=K_{\phi_t}$. Then $K_t$ is differentiable with respect to $t$ . Let $\dot{K_t}$ and $\dot{\phi_t}$ be the derivatives of $K_t$ and $\phi_t$ with respect to $t$. Then for $z$ and $\zeta$ fixed 
\be
\dot{K_t}(z,\zeta)=\int_X \dot{\phi_t}K_t(z,w) K_t(w,\zeta) e^{-\phi_t(w)} d\mu(w).
\ee
Moreover, for the difference quotients we have, if $|\tau|\leq 1$,  
\be
|(K_{t+\tau}(z,z)-K_t(z,z))/\tau|\leq A K_t(z,z)
\ee
for some constant $A$ depending on the sup-norm of $\dot \phi$. 
\end{prop}
\begin{proof}Note first that since $\dot \phi$ is bounded, $\phi_t-\phi_{t+\tau}$ is bounded for $|\tau|\leq 1$.  Hence $e^{-\phi_t}$ and $e^{-\phi_{t+\tau}}$ are of the same magnitude and it  follows from the extremal characterization of Bergman kernels that $K_t$ and $K_{t+\tau}$ are comparable as well. Let $\Delta(t,\tau)= e^{-\phi_t}-e^{-\phi_{t+\tau}}$. Since
$$
\Delta(t,\tau)=\int_0^\tau \dot\phi_s e^{-\phi_{t+s}}ds,
$$
$|\Delta(t,\tau)|\leq A |\tau|e^{-\phi_t}
$
if $|\tau|\leq 1$. Next note that by the reproducing property of Bergman kernels
\be
(K_{t+\tau}-K_t)(z,\zeta)=\int_X K_t(z,w)K_{t+\tau}(w,\zeta)(e^{-\phi_t}-e^{-\phi_{t+\tau}})d\mu(w).
\ee
Hence for $|\tau|\leq 1$
$$
|(K_{t+\tau}(z,z)-K_t(z,z))/\tau|\leq A\int_X |K_t(z,w)K_{t+\tau}(w,z)|e^{-\phi_t}d\mu(w).
$$
Since $\phi_t-\phi_{t+\tau}$ is bounded this is less than 
$$
 A'\left( \int_X |K_t(z,w)|^2 e^{-\phi_t}d\mu(w)+ \int_X |K_{t+\tau}(z,w)|^2 e^{-\phi_{t+\tau}}d\mu(w)\right)\leq A'' K_t(z,z),
$$
so we have proved (2.3). To prove (2.2) is very easy formally, just differentiating under the integral sign, but to prove that this is legitimate we have to work a bit more.  We first multiply (2.4) by its conjugate and integrate with respect to $\zeta$. Letting $f(z,\zeta):=
(K_{t+\tau}-K_t)(z,\zeta)$ we get
$$
\int |f(z,\zeta|^2 e^{-\phi_{t+\tau}}d\mu(\zeta)=
$$
$$
\int K_t(z,w)K_t(w',z)\Delta(t,\tau)(w)\Delta(t,\tau)(w')\int K_{t+\tau}(\zeta,w')K_{t+\tau}(w,\zeta) e^{-\phi_{t+\tau}(\zeta)}d\mu(\zeta)d\mu(w)d\mu(w').
$$
Using the reproducing property of Bergman kernels in the inner integral this is
$$
\int K_t(z,w)K_t(w',z) K_{t+\tau}(w,w')\Delta(t,\tau)(w)\Delta(t,\tau)(w')d\mu(w)d\mu(w').
$$
Next we apply (2.4) to the integral with respect to $w'$ and get
$$
\int f(w,z) K_t(z,w)\Delta(t,\tau)(w) d\mu(w).
$$
Then use that $|\Delta(t,\tau)|\leq |\tau| e^{-\phi_t}$ and apply Cauchy's inequality to get

\be
\int |f(z,\zeta)|^2 e^{-\phi_t}d\mu(\zeta)\leq A |\tau| K_t(z,z).
\ee
We are now finally ready to prove (2.2). By (2.4)
$$
 (K_{t+\tau}-K_t)(z,\zeta)/\tau=\int_X K_t(z,w)K_{t+\tau}(w,\zeta)(e^{-\phi_t}-e^{-\phi_{t+\tau}})/\tau d\mu(w).
$$
By (2.5) we may replace $K_{t+\tau}$ by $K_t$ in the integral. After that we let $\tau$ tend to zero and get (2.2) by dominated convergence. 

\end{proof}

We now turn to the proof of the comparison principle Theorem 2.1. We first claim that we may assume that $\phi-\psi$ is bounded. To see this, let $u:= \psi-\phi$ so that $u\geq -C$. Put $u_0:= \min (u, 0)$, $\psi_0=\phi+u_0$. Then $\psi_0\leq \psi$ and $\psi_0-\phi$ is bounded.  By the extremal characterization of Bergman kernels $K_{\psi_0}(z,z)\leq K_\psi(z,z)$. On the other hand, where $\psi<\phi$, $u<0$ so $u_0=u$. Hence  $\psi_0=\psi$ and $B_{\psi_0}\leq B_\psi$. Moreover $\psi<\phi$ if and only if $u<0$ which is equivalent to $u_0<0$, so $\psi<\phi$ if and only if $\psi_0<\phi$. Hence it suffices to prove the theorem for $\psi_0$ since then
$$
\int_{\psi<\phi} B_{\phi} d\mu\leq\int_{\psi_0<\phi} B_{\psi_0} d\mu\leq \int_{\psi<\phi} B_{\psi} d\mu.
$$ 

From now on we assume that $\phi-\psi$ is bounded and   let $\rho$ be a measurable function on $X$  such that 
$$
\int_X \rho(z)K_\phi(z,z) e^{-\phi} d\mu(z)<\infty. 
$$
  The same integral with $\phi$ replaced by $\psi$ is then also bounded. Let $\phi_t=\phi+tu$, so that $\phi_0=\phi$ and $\phi_1=\psi$. Then we claim that by 
Proposition 2.2, if  
$$
G(t):=\int_X \rho(z) B_{\phi_t} d\mu
$$
then
\be
G'(t)=\int_X -\rho(z)\dot{\phi_t}(z)K_t(z,z)e^{-\phi_t} d\mu + \int_X\int_X \rho(z) \dot{\phi_t}(w) K_t(z,w)K_t(w,z)e^{-\phi_t(z)-\phi_t(w)}d\mu(z)d\mu(w).
\ee
Again, this follows formally by the proposition and to justify the limit process we write
$$
(G(t+\tau)-G(t))= \int_X \rho K_t(e^{-\phi_{t+\tau}}-e^{-\phi_t})d\mu +
\int_X\rho(K_{t+\tau}-K_t)e^{-\phi_{t+\tau}}d\mu.
$$
When we divide by $\tau$ and let $\tau\to 0$ we see that the first term converges to the first term of (2.6) by dominated convergence. For the second term we use  (2.3) to conclude that we have dominated convergence in that integral as well.

In the first integral on the right hand side we insert the reproducing formula for the Bergman kernel
$$
K_t(z,z)=\int_X K_t(z,w)K_t(w,z)e^{-\phi_t(w)}d\mu(w)
$$
which changes the right hand side to
$$
\int_X\int_X
\rho(z) (\dot{\phi_t}(w)-\dot{\phi_t}(z)) K_t(z,w)K_t(w,z)e^{-\phi_t(z)-\phi_t(w)}d\mu(z)d\mu(w).
$$
We can write this more symmetrically as 
\be
(1/2)\int_X\int_X
(\rho(z)-\rho(w)) (\dot{\phi_t}(w)-\dot{\phi_t}(z)) |K_t(z,w)|^2 e^{-\phi_t(z)-\phi_t(w)}d\mu(z)d\mu(w).
\ee
Now recall that $\phi_t=\phi+tu$ so  $\dot{\phi_t}=u$. Let $\rho$ be the characteristic function of the set where $  \psi-\phi=u<0$. Then (2.7) becomes
$$ (1/2)\int\int_{\{u(z)<0<u(w)\}}(u(w)-u(z))|K_t(z,w)|^2 e^{-\phi_t(z)-\phi_t(w)}d\mu(z)d\mu(w)-
$$
$$
-(1/2)\int\int_{\{u(w)<0<u(z)\}}(u(w)-u(z))|K_t(z,w)|^2 e^{-\phi_t(z)-\phi_t(w)}d\mu(z)d\mu(w).
$$
Again using symmetry we get 
\be
\frac{d}{dt}\int_{u<0} B_{\phi_t} d\mu=\int\int_{\{u(z)<0<u(w)\}}(u(w)-u(z))|K_t(z,w)|^2 e^{-\phi_t(z)-\phi_t(w)}d\mu(z)d\mu(w).
\ee
Clearly this expression is non negative, so we have proved Theorem 2.1 under the assumption that 
$$
\int_X \rho(z)K_\phi(z,z) e^{-\phi} d\mu(z)<\infty, 
$$
i e that the left hand side of (2.1) is finite. If this does not hold, clearly the right hand side is infinite as well, so the (in)equality holds trivially.  
\qed

\noindent{\bf Remark:} Since the Bergman function $B_\phi(z)=K_\phi(z,z)e^{-\phi}$ does not change if we add a constant to $\phi$, we also have that
\be
\int_{\psi<\phi+c} B_\phi d\mu\leq\int_{\psi<\phi+c} B_\psi d\mu
\ee
for any constant $c$. One may note that it follows already from (2.7) that the derivative is non negative if $\rho= k(u)$ is an increasing function of $u=\dot\phi$. This fact is however equivalent to (2.9) for all values of $c$. We have chosen to write the derivative in the form (2.8) since this makes it easy to see when we have strict inequality (see the next section). 

\section{The proof of Theorem 1.1}

It is now an easy matter to deduce Theorem 1.1 from Theorem 2.1. First we note that  the setting of line bundles instead of scalar valued functions causes no extra difficulties. Indeed the proof goes through in the same way with only nominal changes. Alternatively, we could use that any line bundle has a discontinuous trivializing section and since continuity played no role in the proof, the line bundle case follows. It remains to prove that we have strict inequality if $B_\psi$ is non trivial and $\emptyset\neq\{\psi<\phi\}\neq X$. For this it suffices to show that the right hand side of (2.8) is strictly positive. But
$$
V:=\{(z,w); u(w)<0<u(z)\}
$$
is by assumption non empty. Moreover, this set is open for the plurifine topology and therefore has positive Lebesgue measure, \cite{BT}.  Hence it has positive $\mu$-measure  if $\mu$ is given by a strictly positive continuous density. From this it follows that if $K_t(z,w)$ is not identically zero, it is different from zero  almost everywhere on $V$, since it is holomorphic with respect to $z$ and $w$ (this is the only place we use holomorphicity). Hence the derivative of $G$ is strictly positive. 

Finally we give a 'maximum principle' for Bergman kernels which follows from Theorem 1.1 in the same way that the Monge-Amp\`ere maximum principle follows from the classical comparison principle.
\begin{thm} Under the same assumptions as in Theorem 1.1, let $\Omega\neq X$ be a subset of $X$ such that $B_\phi\geq B_\psi$ on $\Omega$. Assume that $\phi\leq\psi$ on $X\setminus \Omega$. Then $\phi\leq\psi$ everywhere.
\end{thm}
\begin{proof} Assume the set $U$ where $\psi<\phi$ is nonempty. Then $U$ is a subset of $\Omega$, and $\Omega$ is not equal to $X$. This contradicts Theorem 1.1.
\end{proof}

 \def\listing#1#2#3{{\sc #1}:\ {\it #2}, \ #3.}

\end{document}